\documentclass{amsart}
\usepackage{graphicx} 
\usepackage{amsmath,amsfonts,bm}
\usepackage{latexsym,amssymb,verbatim}
\usepackage{paralist}
\usepackage{ulem}
 \usepackage{graphics} 
  \usepackage{epsfig} 
\usepackage{epstopdf}
 \usepackage[colorlinks=true]{hyperref}
\hypersetup{urlcolor=blue, citecolor=red}
\newtheorem{lema}{Lemma}[section]
\newtheorem{theo}[lema]{Theorem}
\newtheorem{prop}[lema]{Proposition}
\newtheorem{coro}[lema]{Corollary}
\theoremstyle{definition}

\theoremstyle{remark}
\newtheorem{rema}[lema]{Remark}

\newtheorem{hypothesis}[lema]{Hypothesis}
\newtheorem{conj}{Open Question}

\title{Quartic rigid systems in the plane and in the Poincar\'e sphere}
\author{M.J. \'Alvarez}
\address{Departament de Matem{\`a}tiques i Inform{\`a}tica, IAC3 Institute of Applied Computing \& Community Code, Universitat de les Illes Ba\-lears, 07122, Palma de Mallorca, Spain}
\email{chus.alvarez@uib.es}

\author{J.L. Bravo}
\address{Departamento de Matematicas, Universidad de Extremadura, 06071 Badajoz, Spain}
\email{trinidad@unex.es}
\author{L.A. Calder\'on*}
\address{Departament de Ciencies Matematiques i Informatica, IAC3 Institute of Applied Computing \& Community Code, Universitat de les Illes Ballears, 07122 Palma, Spain}
\email{l.calderon@uib.es}


\begin{document}

\begin{abstract}
We consider the planar family of rigid systems of the form $x'=-y+xP(x,y), y'=x+yP(x,y)$, where $P$ is any polynomial with monomials of degree one and three.
This is the simplest non-trivial family of rigid systems with no rotatory parameters. 

The family can be compactified to the Poincaré sphere such that the vector field along the equator is not identically null. We study the centers, singular points and limit cycles of that family on the plane and on the sphere. 

\end{abstract}

\maketitle

\section{Introduction}

In this work we are going to study rigid planar polynomial differential systems. These systems are characterized by  having the origin as its unique critical point, being always monodromic. Furthermore, the solutions around the origin have constant angular velocity. These systems can be written, after a linear change of variables and a rescaling if necessary, see \cite{Conti}, as follows:
\begin{equation}\label{eq:rigido}
\left\{
\begin{aligned}
x'=-y+x F(x,y),\\
y'=\phantom{-}x+y F(x,y),
\end{aligned}
\right.
\end{equation}
for some analytic function $F(x,y)$ such that $F(0,0)=0.$ Observe that, if the origin is a center, then it is isochronous, that is, all the solutions take the same time to complete a full revolution around the origin. In this case,  the center is referred to as uniformly isochronous, as the angular velocity remains constant throughout this rotational motion.

There are several factors that make the rigid family interesting. Firstly,  the fact that the origin is its only critical point implies that any potential limit cycles, if they exist, have to be nested around it. Secondly, this family plays an important role in the broader problem of isochronicity, as we will explain later on. 

In this work, we are going to study the center and cyclicity problems for a family of planar rigid systems. However, the main contribution of this work lies in the study of this system in the Poincaré sphere. Seen in this context, the system exhibits some very interesting properties that, to the best of our knowledge, have not been previously investigated. Understanding the  system's dynamics across the whole sphere provides valuable information about the planar system in both its finite and its infinite parts. As an example, in Theorem \ref{teo-po-esfera} we prove that the system always has  a periodic orbit in the sphere. This periodic orbit can not be seen in, nor from, the finite plane, although this solution plays an important role in the organization of the dynamics of the planar system.

After its introduction by Conti in \cite{Conti}, the rigid systems has attracted the attention of numerous researchers, see for instance \cite{ARJMAA03,  Itikawa_tesis, LlibRamRamSad} and the references therein. In \cite{Rudenok} it was proved that any polynomial system with linear part 
 $(-y,x)^t$ has an isochronous center if and only  it can be transformed by means of a specific analytic change of type $$(x \to x + P(y^2), y\to y + Q(x,y))$$ into a system of the form \eqref{eq:rigido}. Consequently, the problem of determining whether a center is isochronous passes through the understanding of rigid systems. 

Observe that, when the function $F(x,y)$ in \eqref{eq:rigido} is a polynomial of degree $n$,  the system  transformed into polar coordinates can be written as the scalar equation
\begin{equation}\label{eq:Abel_general}
    \dot r=\sum_{k=1}^n F_k(\cos\theta,\sin\theta)r^{k+1},
\end{equation}
being $F(x,y)=\sum_{k=1}^n F_k(x,y).$
The previous equation is a generalized Abel equation and the solutions of \eqref{eq:rigido} are in one-to-one correspondence with positive solutions of \eqref{eq:Abel_general}. In the case where $n=1,$ the rigid system has no limit cycles. This fact  can be easily proved as the scalar equation reduces to a Riccati one of separable variables. However, when $n=2$ there are examples with one limit cycle (observe that the constant term in $F(x, y)$ has been omitted). There is a conjecture suggesting that this is the maximum number of limit cycles that the rigid system can have, see \cite{Gasull}.

One of the properties of the rigid systems is that in some cases, such as when the function $F(x,y)$ is an even-degree polynomial or features a constant term, there is a rotatory parameter, see for instance \cite{GasProTor}. When this rotatory parameter exists, the birth, growth and disappearance of a potential limit cycle is, somehow, controlled.

In this work, we are going to study the simplest family of polynomial rigid systems for which none of its parameters is rotatory. Concretely, the family we are going to study is family \eqref{eq:rigido} wherein the function $F(x, y)$ is defined as follows:
\begin{equation}\label{funcionF}
F(x,y)=b_1 x + b_2 y + a_1 x^3 + a_2 x^2 y + a_3 x y^2 + a_4 y^3\end{equation}

The generalized Abel equation that is in correspondence to the  system we are interested in is
\begin{equation}\label{ec_Abel}
r'=B(\theta)r^2+A(\theta)r^4,\end{equation}
where
\[
\begin{split}
B(\theta) &= b_1 \cos\theta+b_2\sin \theta,\\
A(\theta)&=a_1\cos^3\theta+a_2\cos^2\theta\sin\theta+a_3\cos\theta\sin^2\theta+a_4\sin^3\theta.
\end{split}
\]

Without loss of generality, it is possible to consider the parameter $a_4=0.$ This can be achieved by doing a rotation of angle $\phi,$ being $\phi$ a real root of a specific trigonometric polynomial.

The existence of such a root is guaranteed as the polynomial in question is cubic. Hence, from now on we will consider the family with $a_4=0$ and thus, the concrete family we are studying in this work is
\begin{equation}\label{eq:main}
\left\{
\begin{aligned}
x'=-y+x (b_1 x + b_2 y + a_1 x^3 + a_2 x^2 y + a_3 x y^2 ),\\
y'=\phantom{-}x+y(b_1 x + b_2 y + a_1 x^3 + a_2 x^2 y + a_3 x y^2 ),
\end{aligned}
\right.
\end{equation}

The paper is organized as follows:  Section \ref{Sec2} is dedicated to the study of the center and cyclicity problems in the finite plane. In Section \ref{Sec3} we study the system in the Poincar\'e sphere; more specifically, we study the classification of the infinite critical points and the periodic orbits within the sphere. Finally, in Section \ref{Sec4} we make some conclusions and conjectures about system \eqref{eq:main}.

\section{Centers and cyclicity}\label{Sec2}

In the existing literature, rigid systems having a center at the origin are  commonly referred to as uniform isochronous centers, as their angular velocity is constant.
The center conditions for a general rigid system were given in  \cite{ARJMAA03} in terms of the existence of an analytic commutator. Furthermore, in the same paper, the authors proved that the rigid family with $F(x,y)=F_1(x,y)+F_m(x,y),$ where $F_m$ is a homogeneous polynomial of degree $m,$ has a center at the origin if and only if it is reversible. Moreover, in \cite{Itikawa} the 14 different phase portraits of quartic uniform isochronous centers are given.  In the next result we give the explicit conditions in terms of the parameters of the system to have a center at the origin.

\begin{theo}\label{theo:center}
The origin of system \eqref{eq:main} is a center if and only if one of the following conditions is satisfied:
    \begin{enumerate}
\item $b_1=b_2=0$, 
\item 
$
3 a_1 b_2 (b_2^2-b_1^2)+
b_1 (a_2 b_1^2 + 2 a_3 b_1 b_2 - 3 a_2 b_2^2)=0$ and \\ 
$b_2(-3 a_3 b_1^2 + 2 a_2 b_1 b_2 + a_3 b_2^2)=0.
$
\end{enumerate}

\end{theo}

\begin{proof}
We compute the first Lyapunov constants of the origin of system \eqref{eq:main}, getting:
\begin{eqnarray*}
&&l_2=\frac{\pi}{2} (a_2 b_1 - 3 a_1 b_2 - a_3 b_2),\\
&&l_3=-\frac{\pi}{2} (-a_2 b_1^3  + 3 a_1 b_1^2 b_2 + 3 a_3 b_1^2 b_2 - 9 a_2 b_1 b_2^2 + 23 a_1 b_2^3 + 7 a_3 b_2^3). 
    \end{eqnarray*}
If we set them to zero, we get the two conditions in the statement of the
theorem. It remains to prove that in both cases these conditions imply that the origin
is a center.
\begin{enumerate}
\item In this case the function $F(x,y)$ is homogeneous of degree 3. Note that in this case, the system is integrable, and one first integral is
\[
H(x,y) =
\frac{-1 + a_2x^3 - 3 a_1 x^2 y  - (2 a_1 + a_3) y^3}{3\sqrt{(x^2 + y^2)^3}}.
\]

\item In this last case the system is reversible with respect to the straight line $b_1x+b_2y=0.$ 
\end{enumerate}
Hence, the result is proved.
\end{proof}
\begin{rema}
    Observe that the trivial cases $b_1=a_1=a_3=0$ and $b_2=a_2=0$ are included in the second family of the previous result. These two subfamilies are reversible with respect to the straight lines $y=0$ and $x=0,$ respectively.
\end{rema}

The center conditions are closely related  to the order of weakness of the focus and, hence, with the cyclicity of the critical point. In  \cite{ARJMAA03}, it is proved that the maximum order of a  fine focus of the rigid system with $F(x,y)=F_1(x,y)+F_m(x,y)$ is $\lfloor m/2\rfloor+1,$ that in our case is 2.  In the following result we prove that this fact leads to the existence of at least one limit cycle inside the family \eqref{eq:main}.

\begin{prop}
    There are systems inside the family \eqref{eq:main} having at least one limit cycle.
\end{prop}
\begin{proof}
Consider next system inside the family \eqref{eq:main}:
  \begin{equation}\label{eq:cl}\left\{
\begin{aligned}
x'=-y+x\left(5x+y + \frac{1+120a_2\pi -82\varepsilon}{74\pi} x^3 + a_2 x^2 y + \frac{-3+10a_2\pi+98\varepsilon}
{74 \pi}x y^2\right),\\
y'=\phantom{-}x+y\left(5x+y + \frac{1+120a_2\pi -82\varepsilon}{74\pi} x^3 + a_2 x^2 y + \frac{-3+10a_2\pi+98\varepsilon}
{74 \pi}x y^2\right).
\end{aligned}
\right.
\end{equation}
Doing some simple computations, one can prove that its Lyapunov constants are
\[
l_2=\varepsilon,\quad l_3=1.
\]
Choosing $\varepsilon<0$ small enough, one limit cycle is born from the origin by a degenerate Hopf bifurcation.    
\end{proof}
\begin{rema}
Although the order of weakness of the focus is two, no more limit cycles can be created by a Hopf bifurcation inside the family \eqref{eq:main}. This is because for the family we are studying the divergence of the system is zero.
\end{rema}

Concerning the number of limit cycles that system \eqref{eq:main} can have, in \cite{GasTor} the authors
studied rigid systems with $F(x,y)=F_0(x,y)+F_m(x,y)+F_n(x,y),$ being $F_k$ a homogeneous polynomial of degree $k$. For low degrees of $m,n$ they found lower bounds for the number of limit cycles. The best result they obtained for $F_0(x,y)\equiv 0$ and $m=1$ (which is the case of system \eqref{eq:main}) was 1, which matches what we have obtained in the previous result. 

 As it has been mentioned in the introduction, the solutions of the family \eqref{eq:main} are in one-to-one correspondence with the solution of the Abel equation \eqref{ec_Abel}. In the following result we  use a specific criteria for Abel equations  in order to bound the number of limit cycles for a subfamily of system \eqref{eq:main}.
 
For this result, instead of setting the parameter $a_4$ equal to zero, it is more convenient to fix $b_1=0$. This can be done without loss of generality by doing a rotation of $\theta$. In the rest of the paper we will be working with $a_4=0$, this is, with \eqref{eq:main}.

\begin{theo} Consider the family \eqref{eq:rigido} with the function $F$ being the one defined in \eqref{funcionF}, for which it is not restrictive to assume $b_1=0,$ that is, the family
\begin{equation*}
\left\{
\begin{aligned}
x'=-y+x (b_2 y + a_1 x^3 + a_2 x^2 y + a_3 x y^2+a_4y^3 ),\\
y'=\phantom{-}x+y(b_2 y + a_1 x^3 + a_2 x^2 y + a_3 x y^2+a_4y^3).
\end{aligned}
\right.
\end{equation*}
If $a_1 a_3\geq 0$ then 
the system has no limit cycles.   
\end{theo}
\begin{proof}
We transform the system into the Abel equation \eqref{ec_Abel} and denote 
\[
\begin{split}
f(\theta) &=  \cos(\theta) (a_1\cos^2(\theta)+a_3\sin^2(\theta)),\\
g(\theta) &= \sin(\theta) (a_2\cos^2(\theta)+a_4\sin^2(\theta)),\\
h(\theta) &= b_2 \sin(\theta).
\end{split}
\]
Now, applying Theorem~2.4 of \cite{BT} we conclude. 
\end{proof}

\section{Vector field on the Poincaré sphere}\label{Sec3}

In order to understand the full behaviour of a planar system, it is necessary to look at the solutions that approach or escape from infinity.  To do this, one must compactify the plane. There are different compactifications in the literature, and the most common ones are those that allow infinity to be seen as a point or as a
circle.  In this paper we will use the Poincaré compactification, which transforms the planar system into two copies of a vector field but now defined in the sphere. These two copies are separated by the equator of the sphere, which contains the information about the dynamics at infinity. For a more detailed description of this compactification see \cite{Dumo}.

In our case, as in all systems having its highest degree as a radial one (see \cite{BravoFerTer}), the infinity is a circle of singularities.  Therefore, we have to apply a slight  modification of the classical Poincar\'e compactification in order to deal with it, see again \cite{BravoFerTer}.

We exclude from our study the case where $a_1=a_2=a_3=0,$ since in this case, as we have mentioned before, the resulting system is a Riccati equation of separable variables and it can be easily integrated.

\subsection{Poincaré compactification}

We consider the real plane embedded in $\mathbb{R}^3$ as the tangent plane to $\mathbb{S}^2$ in the north pole. Consequently, the points of the plane are of the form $(x,y,1)$. We can project each point of the plane over the upper half sphere taking the straight line between the point and the center of the sphere, this is, if $(z_1,z_2,z_3)$ are the coordinates of the assigned point in the sphere, then $x=z_1/z_3$ and $y=z_2/z_3$, or conversely
\[
z_1=\frac{x}{\Delta},\quad 
z_2=\frac{y}{\Delta},\quad 
z_3=\frac{1}{\Delta},
\]
where $\Delta=\sqrt{x^2+y^2+1}$. The projection in the south hemisphere is the same, changing the sign of every component of the field. 

Using this construction, we can  project the vector field in the plane given by \eqref{eq:main} over each of the hemispheres. 

\begin{prop}

The vector field \eqref{eq:main} is topologically equivalent to the restriction to the northern hemisphere of the system

\begin{equation}\label{eq:campoesfera}
\begin{cases}
z_1'&=z_3 \left(-z_2 z_3 + b_1 z_1^2 z_3^2 + b_2 z_1 z_2 z_3^2  + a_1 z_1^4   + a_2 z_1^3 z_2 + a_3 z_1^2 z_2^2\right)\\
z_2'&=z_3 \left(z_1 z_3 + b_1 z_1 z_2 z_3^2 + b_2 z_2^2z_3^2  + a_1 z_1^3 z_2 + a_2 z_1^2 z_2^2  + a_3 z_1 z_2^3  \right)\\
z_3'&=\left(z_3^2-1\right)  \left(b_1 z_1 z_3^2 + b_2 z_2 z_3^2 + a_1 z_1^3 + a_2 z_1^2 z_2 + a_3 z_1 z_2^2\right).
\end{cases}
\end{equation}

\end{prop}
\begin{proof}
Consider the projection of the plane $(x,y,1)$ to the Poincaré sphere defined by 
$z_1 = x/\Delta$,
$z_2 = y/\Delta$,
$z_3 = 1/\Delta$,
where $\Delta=\sqrt{x^2+y^2+1}$ and $z_1^2+z_2^2+z_3^2=1$. Now, deriving the projection along the solutions we obtain that the projection of the vector field \eqref{eq:main} to the northern hemisphere is 
\begin{equation*}
\begin{cases}
z_1'&= (1-z_1^2) P(z_1/z_3,z_2/z_3) - z_1z_2 Q(z_1/z_3,z_2/z_3) \\
z_2'&= -z_1z_2 P(z_1/z_3,z_2/z_3) + (1-z_2^2) Q(z_1/z_3,z_2/z_3)\\
z_3'&= -z_3\left(z_1 P(z_1/z_3,z_2/z_3) + z_2 Q(z_1/z_3,z_2/z_3) \right).
\end{cases}
\end{equation*}
Reparametrizing time by a factor $z_3^3$  and using that $z_1^2+z_2^2+z_3^2=1$, we obtain \eqref{eq:campoesfera}.

\end{proof}
\subsection{Critical points at infinity}

In order to study the infinite critical points, we have to use local charts. Since we have taken the parameter $a_4=0$, we have a common factor $x$ in the highest degree of the original system which becomes a common factor $z_1$ in the highest degree of the system in the sphere; therefore the chart that will better suit the study of the infinite critical points, in order to see all of them, will be  $\mathcal{U}_2.$

Remember that here the infinite critical points are present in symmetric pairs. This is, after we know how many critical points we have in the chart $\mathcal{U}_2$, the total number of infinite critical points of the system will be the double of this amount, as for each point in the chart $\mathcal{U}_2$ we have its symmetric. 

In the chart $\mathcal{U}_2$, the expression of the variables  is $(x,y)=(\frac{u}{v},\frac{1}{v})$ and the system, after a re-scaling of $v^4,$ is
\begin{equation*}\label{U2}
\begin{cases}
u'&=-v^2(u^2 + 1),\\
v'&=-a_1u^3 - b_1uv^2 - uv^3 - a_2u^2 - b_2v^2 - a_3u.
 \end{cases}
 \end{equation*}
The infinite critical points at this chart will be $(\hat{u},0)$, with $\hat{u}$ being a root of the cubic $g(u)=-u(a_1 u^2+a_2 u+a_3)$. Consequently, there will be as many infinite critical points in this chart as roots of $g(u)$. 

For the rest of this section, when we refer to "simple", "double" or "triple" critical points, we will just make reference to the multiplicity of $\hat{u}$ as a root of $g(u)$, and not to their description as critical points. 

Note that, if $a_1=0$, $a_2\neq0$ or if $a_1=a_2=0$, we are no longer able to see one or two of the infinite critical points in the chart $\mathcal{U}_2$, respectively. If we study the chart $\mathcal{U}_1,$ the critical points will be $(\tilde{u},0)$ with $\tilde{u}$ being a root of $h(u)=-(a_3 u^2+a_2 u+a_1)$. Observe that in these cases the critical points that cannot be seen in the chart $\mathcal{U}_2$ will correspond with the root zero of $h(u)$ in the chart $\mathcal{U}_1$.

We will divide the study of the infinite critical points in two main cases: if they are all simple or if not. If $a_1\neq0$, note that, if we denote by $D=a_2^2-4a_1a_3$ the discriminant of the quadratic factor of $g(u)$, then the infinite critical points will be simple if and only if  $a_3\, D\neq0$. 

Again, if $a_1=0$ the critical points that we can no longer see in the chart $\mathcal{U}_2$ correspond to the root zero of $h(u)$ in the chart $\mathcal{U}_1$. The only case in which all the roots are simple in both charts is $a_2\,  a_3\neq0$.

\begin{prop}
Consider family \eqref{eq:campoesfera}. If all the infinite critical points of the system are simple, then they are cusps.

More concretely, denoting $D=a_2^2-4a_1a_3$, \eqref{eq:campoesfera} has only simple critical points if and only if  $(a_1^2+a_2^2)\, a_3\, D\neq0$. Moreover,
\begin{itemize}
    \item If $D>0$, there are three infinite critical points in the chart $\mathcal{U}_2$, which are cusps.
    \item If $D<0$, there is only one infinite critical point in the chart $\mathcal{U}_2$, which is a cusp.
\end{itemize}

\end{prop}

\begin{proof}
    Assume first that $a_1\neq0$. Either $D>0$ or $D<0$. We remind that in this case the infinite critical points in the chart $\mathcal{U}_2$ are $(\hat{u},0)$, with $\hat{u}$ being a root of the cubic $g(u)=-u(a_1 u^2+a_2 u+a_3)$, and that $D$ is the discriminant of the quadratic factor of $g(u)$.

    If $D>0$, the three simple critical points are nilpotent ones.  We proceed to apply Andreev Nilpotent Theorem (see \cite{Andreev} or \cite{Dumo}). We will do the computations for $(0,0)$ and the other two critical points  follow in a similar way.

First of all, we do a re-scaling and we interchange the name of the variables, in order to transform the system in its nilpotent normal form. Hence, $\tilde{u}=v, \tilde{v}=u$ and $t=-a_3s.$ Remember that in this case $a_3\neq0$. Now, dropping the tildes and denoting by a dot the derivative with respect to $s,$ the system is transformed into
\[
\begin{cases}
   \dot{u}&=v+\frac{a_1v^3 + b_1u^2v + u^3v + a_2v^2 + b_2u^2}{a_3},\\\dot{v}&=\frac{u^3(v^2 + 1)}{a_3}.
\end{cases}
\]
We denote $\dot{u}=v+A(u,v), \dot{v}=B(u,v).$ 
If we solve the equation $v+A(u,v)=0,$ we get the
 function $v=f(u)=-\frac{b_2}{a_3}u^2+O(u^3).$
 
 Now, we can compute the functions $F(u)=B(u,f(u))$ and\\ $G(u)=\left(\frac{\partial A}{\partial u}+\frac{\partial B}{\partial v}\right)(u,f(u)),$ and we get
\[\begin{cases}
F(u)&=\frac{u^2}{a_3}+O(u^3),\\
G(u)&=\frac{2b_2}{a_3}u+O(u^2).
\end{cases}\]
Thus, in the Andreev Nilpotent Theorem, we get that $m=2, n=1.$ Consequently, this critical point is a cusp point.

We can proceed in the same way with the other two critical points, getting the same result: in this case, all the infinite critical points are cusps.

If $D<0$, there exists only one infinite critical point, $(0,0),$ being also a nilpotent one. We can proceed exactly in the same way as before, getting that the infinite critical point is also a cusp.

If $a_1=0$, we would have to study the two simple roots of $g(u)$ in the chart $\mathcal{U}_2$, and the simple root zero of $h(u)$ in the chart $\mathcal{U}_1$.

Following an equivalent procedure that in the previous cases, we get that the two simple points in $\mathcal{U}_2$ are nilpotent and, applying again Andreev Nilpotent Theorem, they are cusps. 

Concerning the study of the root of $h(u)$ in the other chart, note that the case $a_1=0$ is equivalent to the case $a_4=0$ by swapping $x$ and $y$ in the original system. Consequently, the computations would be equivalent to the ones already done, and this point will be a cusp. 
\end{proof}

We focus now on the case in which we have infinite critical points in the chart $\mathcal{U}_2$ with multiplicity greater than one. With a suitable change of variables, we can assume that $u=0$ is the root of greater multiplicity of $g(u)$, with new parameters $\bar{b}_1,\bar{b}_2,\bar{a}_1,\bar{a}_2,\bar{a}_3$. Assume first that $\bar{a}_1\neq0$.

Then note that $g(u)$ will have $u=0$ as a double root if and only if $\bar{a}_3=0$, $\bar{a}_2\neq0$, and as a triple one if and only if $\bar{a}_2=\bar{a}_3=0$. 

Assume now that $\bar{a}_1=0$. The only option in which there are critical points with multiplicity greater that one is $\bar{a}_3=0$, $\bar{a}_2\neq0$, when we have a double critical point in the chart $\mathcal{U}_2$ and a simple one in the chart $\mathcal{U}_1$. Note that if $\bar{a}_2=0$, $\bar{a}_3\neq0$, the double point would be in the chart $\mathcal{U}_1$ and not in the zero of chart $\mathcal{U}_2$, which is not possible after our change of variables. 

\begin{prop}
Consider family \eqref{eq:campoesfera} with infinite critical points with multiplicity greater than one, and after the change of variable described above. The infinite critical points of the system are either cusps or consist of two parabolic and two hyperbolic sectors.

More concretely, if $\bar{a}_1\neq0$:
    \begin{itemize}
        \item If $\bar{a}_3=0$, $\bar{a}_2\neq0$, the system in the chart $\mathcal{U}_2$ has a simple critical point, which is a cusp, and a double one that has two hyperbolic and two parabolic sectors. 
        \item If $\bar{a}_2=\bar{a}_3=0$, the system in the chart $\mathcal{U}_2$ only has a triple critical point, which has two hyperbolic and two parabolic sectors when $\bar{b}_2\neq0$, and is a cusp when $\bar{b}_2=0$.
    \end{itemize}
If $\bar{a}_1=0$, in which case the only possibility is $\bar{a}_3=0$, $\bar{a}_2\neq0$, the system in the chart $\mathcal{U}_2$ has a double critical point, which has two hyperbolic and two parabolic sectors, and a simple one that  is a cusp. 
        
\end{prop}

\begin{proof}

For clarity, we will organize this proof in cases and subcases. Only some of them will be presented with full detail, as the other ones follow in a similar way. 

$\boldsymbol{\textbf{Case 1. }\bar{a}_1\neq0.}$ We will divide this case depending on the multiplicity of the root zero of $g(u)$.  

$\boldsymbol{\textbf{Case 1.1. }\bar{a}_3=0, \bar{a}_2\neq0}$. This is, we have a double critical point at $(0,0)$ and another simple critical point.

In this case, following a similar reasoning as in the previous result, the simple critical point is a cusp. In order to study the double point, we make a directional blow-up, $u=u_1, v=u_1v_1.$ After the change of time $s=t/u_1,$ the system is
\begin{equation*}
\begin{cases}
    u_1'=\frac{du_1}{ds}= -u_1(1+u_1^2)v_1^2,\\
v_1'=-\bar{a}_2-\bar{a}_1u_1-\bar{b}_2v_1^2-\bar{b}_1u_1v_1^2+v_1^3.
\end{cases}
\end{equation*}
We need to study the critical points on $u_1=0:$
\[
v_1'\Big|_{u_1=0}= -\bar{a}_2-\bar{b}_2v_1^2+v_1^3=:\hat{g}(v_1).
\]
This cubic polynomial will have one, two or three roots depending on the sign of the discriminant $-\bar{a}_2(27\bar{a}_2+4\bar{b}_2^3).$ It is very useful to work from this point ahead in terms of the roots of $\hat{g}(v_1)$ to simplify the analysis of the system after the blow up. 

Let $\alpha_1,\alpha_2,\alpha_3$ be the roots of $\hat{g}(v_1)$. First, note that if $\alpha_2=-\alpha_3,$ having into consideration the well known relationships between the coefficient and roots of a cubic, necessarily $\alpha_2=\alpha_3=0,$ implying $\bar{a}_2=0,$ which is not our current case. Consequently, we can assume $\alpha_2\neq -\alpha_3.$ 

In this case, by the usual relationships between the roots of a cubic, it follows that $\alpha_1=-\alpha_2\alpha_3(\alpha_2+\alpha_3)^{-1}.$ If the cubic had a triple root, then it would be located at zero, so $\bar{a}_2=\bar{b}_2=0,$ which again is not the case we are currently studying. Consequently, the critical points are either double or simple, so two new subcases appear. 

$\boldsymbol{\textbf{Case 1.1.1. } \hat{g}(v_1)\textbf{ has three simple roots.}}$ In this case we can pick $\alpha_2$ and $\alpha_3$ as new coefficients of the system with the change of parameters
\[\bar{a}_2=-\alpha_2^2\alpha_3^2(\alpha_2+\alpha_3)^{-1},\quad \bar{b}_2=(\alpha_2^2+\alpha_2\alpha_3+\alpha_3^2)(\alpha_2+\alpha_3)^{-1}.\]

When studying the eigenvalues of the jacobian matrix at the three roots, we get that:

\begin{enumerate}
    \item The eigenvalues of the jacobian matrix at $\alpha_2$ are $-\alpha_2^2$ and $\alpha_2(\alpha_2-\alpha_3)(\alpha_2+2\alpha_3)(\alpha_2+\alpha_3)^{-1}$.
    \item The eigenvalues of the jacobian matrix at $\alpha_3$ are $-\alpha_3^2$ and $-\alpha_3(\alpha_2-\alpha_3)(2\alpha_2+\alpha_3)(\alpha_2+\alpha_3)^{-1}$.
    \item The eigenvalues of the jacobian matrix at $-\alpha_2\alpha_3(\alpha_2+\alpha_3)^{-1}$ are $-\alpha_2^2\alpha_3^2(\alpha_2+\alpha_3)^{-2}$ and $\alpha_2\alpha_3(2\alpha_2+\alpha_3)(\alpha_2+2\alpha_3)(\alpha_2+\alpha_3)^{-2}.$ 
\end{enumerate}

If any of these eigenvalues is zero, we can conclude that two of the three roots coincide, which is not our current case. Thus every critical point has a negative eigenvalue and, studying the sign of the rest of them, we get that two of them must be positive, and the other one, negative. 

Consequently, two of the critical points are saddle points and the other one an attractor node. Regardless of their location, undoing the blow up we conclude that the original double critical point has two hyperbolic sectors and two parabolic ones. 

Observe that we have to make the blow up in the other direction ($u_2=uv,v_2=v$) to ensure that there are no orbits arriving at the critical point tangent to the vertical direction. Some simple computations show that this is not the case, and no orbit arrives tangent to the vertical axis.

$\boldsymbol{\textbf{Case 1.1.2. } \hat{g}(v_1)\textbf{ has a double root and a simple one.}}$ Here, making a similar study as in the previous case, the double point must be a saddle-node and the simple one, a saddle point. Undoing  the blow up, the original double critical point has two hyperbolic sectors and two parabolic ones, as in the previous case.

Again, no orbit arrives to the critical point tangent to the vertical axis after doing the blow up in the other direction.

$\boldsymbol{\textbf{Case 1.2. }\bar{a}_2=\bar{a}_3=0}$. This is, the only critical point is triple. We have to make the same directional blow up, getting the system
\begin{equation*}
\begin{cases}
u_1'=\frac{du_1}{ds}= -u_1(1+u_1^2)v_1^2,\\
v_1'=\frac{dv_1}{ds}=-\bar{a}_1u_1-\bar{b}_2v_1^2-\bar{b}_1u_1v_1^2+v_1^3.
\end{cases}
\end{equation*}
We study the critical points on $u_1=0,$ this is, the roots of
\[
v_1'\Big|_{u_1=0}=v_1^2(-\bar{b}_2+v_1).
\]
Here, we have two different subcases. 

$\boldsymbol{\textbf{Case 1.2.1. }\bar{b}_2\neq0}$. In this case, there is a double critical point corresponding to $v_1=0$, and a simple one for $v_1=\bar{b}_2$. The simple critical point is a saddle point while the double one is nilpotent. We follow the usual procedure with the Andreev normal form and conclude that the double point is a saddle-node. Undoing the blow up, again we get that the triple critical point has two hyperbolic sectors and two parabolic ones.

As with all these blow ups, we also have to do the other directional blow up, and, as previously, the computations show that no orbit arrives to the origin with vertical slope. 

$\boldsymbol{\textbf{Case 1.2.2. }\bar{b}_2=0}$. In this case, there is a triple critical point corresponding to $v_1=0$, making the point very degenerate. Hence, it is necessary to perform  some additional blow ups in order to desingularize the critical point. After all the process, and proceeding as in the previous cases, we conclude that the critical point is a cusp.

$\boldsymbol{\textbf{Case 2. }\bar{a}_1=0.}$ We recall that the only possibility in order to have multiple critical points is $\bar{a}_3=0, \bar{a}_2\neq0.$ Studying the double critical point in the chart $\mathcal{U}_2$ in the same way that we have done in the case 1.1.2., we get that this point has two parabolic and two hyperbolic sectors. 

Finally, as we reasoned in the previous result, the computations for the simple root zero in the chart $\mathcal{U}_1$ are equivalent to the ones already done, and the point is a cusp. 
\end{proof}

\begin{coro}\label{resumeninf}
    The infinite critical points of family \eqref{eq:campoesfera} are either cusps, or have two hyperbolic and two parabolic sectors.
\end{coro}

\subsection{Centers in the sphere}

Now, let us consider the centers of the system in the sphere.

\begin{theo}\label{theo:centersphere}
If \eqref{eq:main} has a center at the origin, then it is a global center of the vector field \eqref{eq:campoesfera} on the sphere.
\end{theo}

\begin{proof}
By Theorem~\ref{theo:center}, the origin is a center if and only if one of the following conditions is satisfied:
\begin{enumerate}
\item $b_1=b_2=0$, 
\item 
$
3 a_1 b_2 (b_2^2-b_1^2)+
b_1 (a_2 b_1^2 + 2 a_3 b_1 b_2 - 3 a_2 b_2^2)=0$ and \\ 
$b_2(-3 a_3 b_1^2 + 2 a_2 b_1 b_2 + a_3 b_2^2)=0.
$
\end{enumerate}

Let us check that both conditions are also global centers in the sphere. 

\begin{enumerate}
\item In this case, the first integral of the planar system extends to the following first integral of the system on the sphere
\[
H(z_1,z_2,z_3) =
\frac{-z_3^3 + a_2z_1^3 - 3 a_1 z_1^2 z_2  - (2 a_1 + a_3) z_2^3}{3\sqrt{(z_1^2 + z_2^2)^3}}.
\]
It can be checked simply by differentiating $H$ along the solutions of \eqref{eq:campoesfera}. 

\item In this case the system is reversible with respect to the meridian determined by $b_1z_1+b_2z_2=0$. 
\end{enumerate}
\end{proof}

An open question is whether the converse is true, that is, if it is possible or not to have an annulus formed by a continuum of periodic solutions crossing the equator without having a global center. As we have not been able to prove or disprove it, but we conjecture that it is true, we will assume it as a hypothesis for the rest of the paper. 

\begin{hypothesis}\label{hypo}
Every annulus of periodic orbits of~\eqref{eq:campoesfera} is a global center.
\end{hypothesis}

\subsection{Geometry of the vector field on the sphere}

Consider the vector field extended to the whole sphere by \eqref{eq:campoesfera}. In this case, the system has the two poles as the only critical points outside  the equator, corresponding to the critical point at $(0,0)$ of the original system. Furthermore, the vector field on the sphere has a symmetry with respect to the origin of coordinates. 

\begin{prop}\label{prop:sym}
Assume that  $t\to(z_1(t),z_2(t),z_3(t))$ is a solution of ~\eqref{eq:campoesfera} for certain functions $z_1,z_2,z_3$. Then $t\to -(z_1(t),z_2(t),z_3(t))$ is also a solution.
\end{prop}
\begin{proof}
The proof follows by direct computation.
\end{proof}

As a consequence, both poles have the same local phase portrait. Recall that the origin of the original rigid system is always monodromic, so both poles are also monodromic, and by the symmetry, with the same stability and opposite orientation.

A second consideration is that solutions intersect the equator, $Q$, orthogonally.

\begin{prop}
The vector field is orthogonal to any regular point on the equator. Moreover, if $(z_1,z_2,0)$ is a point on the equator, the direction of the vector field is determined by the sign of $a_1z_1^3+a_2z_1^2z_2+a_3z_1z_2^2$.
\end{prop}
\begin{proof}
It suffices to note that the system~\eqref{eq:campoesfera} at any point on the equator is 
\[
\begin{cases}
z_1'&=0,\\
z_2'&=0,\\
z_3'&=a_1z_1^3+a_2z_1^2z_2+a_3z_1z_2^2.
\end{cases}
\]
\end{proof}

We remind that, as seen in Section 3.2, the system on the equator either has two cusps; six cusps; two critical points with two hyperbolic and two parabolic sectors; or two cusps and two critical points with two hyperbolic and two parabolic sectors.

In the previous section we have studied a subfamily of system \eqref{eq:main} which has no finite limit cycles. Nevertheless, it turns out that when the system has either two or six cusps (and no other critical points) on the equator, the system in the Poincar\'e sphere, system \eqref{eq:campoesfera}, has either a periodic solution or a homoclinic or heteroclinic connection. We divide the result in two cases, corresponding with the number of cusps, two or six. 

\begin{theo}\label{teo-po-esfera}
If system \eqref{eq:campoesfera} has two cusps and no other critical points on the equator, then it always has a periodic solution in the sphere,  symmetric with respect to the origin and its intersection with the equator consists of two (symmetric) regular points.

\end{theo}

\begin{proof} Consider the vector field \eqref{eq:campoesfera} defined on the sphere and denote the equator, $z_3=0$, by $Q.$ A meridian is a great circle joining the two poles, $(0,0,1)$ and $(0,0,-1)$. Note that the vector field is transversal to the meridians except at the poles and the equator. Moreover, the solutions rotate in both hemispheres clockwise.

If the vector field has two simple critical points, then the vector field has two changes of direction along the equator. Without loss of generality, we can assume that these changes are at $(1,0,0)$ and $(-1,0,0).$
We will use the following notation:
\[
Q^+=\{(z_1,z_2,0)\in Q,\colon z_2>0\},
\quad Q^-=\{(z_1,z_2,0)\in Q,\colon z_2<0\}.
\]

\medskip 

{\bf Claim. All solutions starting in $Q^+$ or all solutions starting in $Q^-$ intersect again $Q$.}
Assume that not all solutions starting in $Q^+$ intersect $Q$. Then there exists $p\in Q^+$ such that if $u(t)$ is the solution of \eqref{eq:campoesfera} with initial condition $u(0)=p$, then $u(t)\not\in Q$ for $t>0$. 

Consider any point $q\in Q^-$, and the solution $v(t)$ starting in $q$.  For negative time, if $v(t)$ does not cross the equator, then it is contained in the region $R$ limited by $Q$, $u(t)$, and the meridian passing by $q$.  Moreover, its $\alpha$-limit set must contain a critical point, as in the interior of $R$ there are neither critical points nor limit cycles. 

\begin{figure}[h]
\includegraphics[scale=.25]{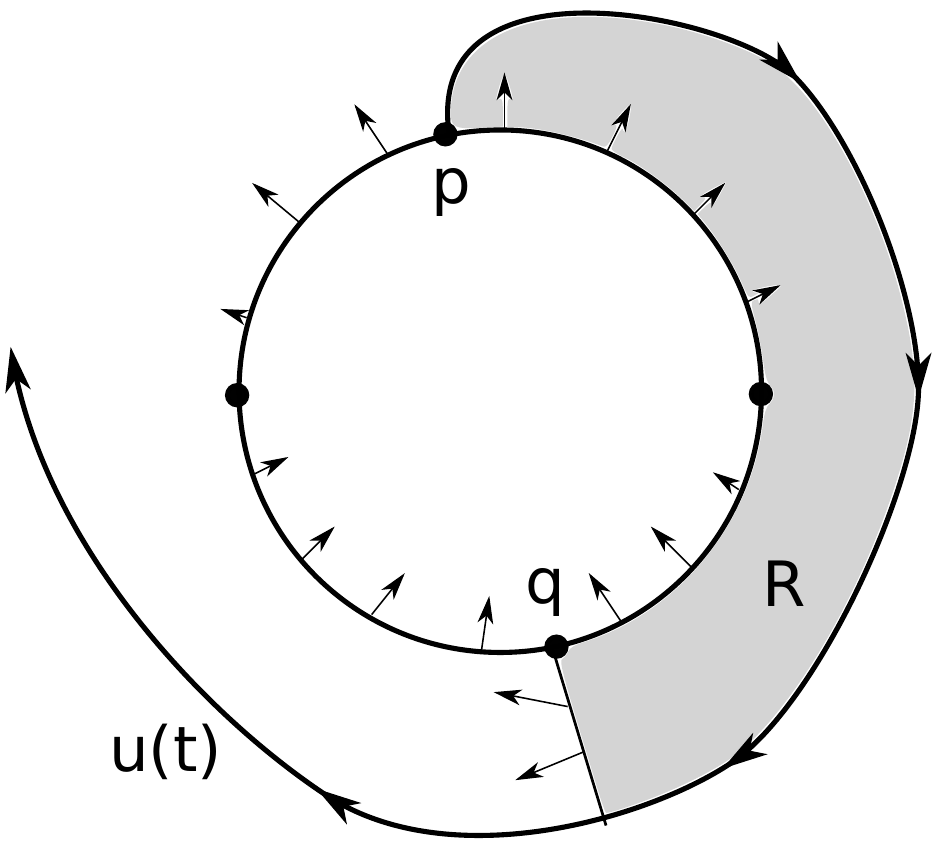}
\caption{Region $R$ in stereographic projection.}
\end{figure}

The unique critical point in $R$ is $(1,0,0)$, but if if belongs to the $\alpha$-limit set, then $v(t)$ and $Q^-$ limit a negative invariant region, so $(1,0,0)$ has a nodal sector, in contradiction with Corollary~\ref{resumeninf}. Therefore, all solutions starting in $Q^-$ intersect $Q^+$ and the claim follows. 

Therefore, we have a map from $Q^+$ to $Q^-$ (or the reverse). Composing this map with the symmetry with respect to the center of the sphere, we have a map from $Q^+$ to $Q^+$. By Brower's fixed point theorem, this map has a fixed point which is a solution crossing $Q$ at symmetric points. By Proposition~\ref{prop:sym}, we conclude this solution is periodic. 

\end{proof}

The previous result has been proved under the hypothesis that the system \eqref{eq:campoesfera} has exclusively two cusps on the equator. In the following result we will prove that when having six cusps on the equator the system also has a periodic solution, provided that there are not homoclinic nor heteroclinic connections.

\begin{theo}\label{teo-po-esfera2}
If system \eqref{eq:campoesfera} has six cusps on the equator and there are no homoclinic nor heteroclinic connections, then it always has a periodic solution in the sphere,  symmetric with respect to the origin and its intersection with the equator consists of two or six (symmetric) regular points.
\end{theo}
\begin{proof}
Consider the vector field \eqref{eq:campoesfera}. In this case, we assume that the infinite critical points, so changes of directions of the vector field at the equator $Q$, are at clockwise-ordered points $p_1$, $p_2$, $p_3$, $p_4$, $p_5$, $p_6$, with $p_1=(-1,0,0)$ (see Figure~\ref{fig:3}). 
Moreover, by Corollary~\ref{resumeninf}, they are cusps. Note that the stable and unstable varieties of consecutive cusps must be in opposite hemispheres. Denote $u_i$ to the unstable variety of $p_i$ and $s_i$ to the stable variety. 

\begin{figure}[h]
\includegraphics[scale=.45]{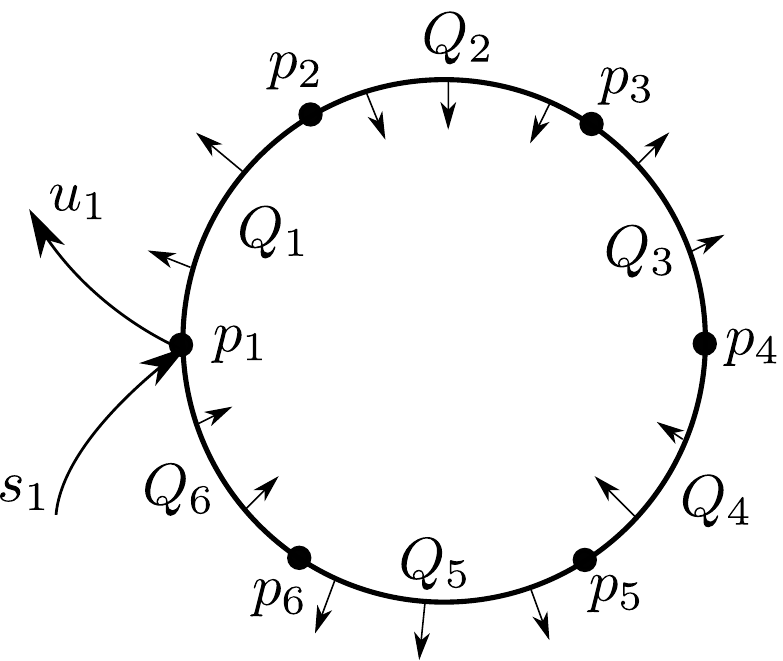}
\caption{Critical points and sectors in stereographic projection.}\label{fig:3}
\end{figure}

We divide the equator in sectors $Q_1$, $Q_2$, $Q_3$, $Q_4$, $Q_5$, $Q_6$, in clockwise order, where $Q_1$ is limited by $p_1$ and $p_2$, and so on. 

Note that the points $p_i,p_j$ are symmetric if $|i-j|=3$, and the same holds for $Q_i,Q_j$. Moreover, as we have a cyclic ordering, we will work with all the indices in $\mathbb{Z}/6\mathbb{Z}$.

As the solutions of the vector field always rotate clockwise, we will consider the intersections of $u_i,s_i$ with $Q$ in the first turn around the center. Note that by the directions of the vector field, $u_i$ can not cut $Q_i$ in its first turn and $s_i$ can not cut $Q_{i-1}$ in its first turn.

\medskip 

{\bf Claim 1. There exists $i\in\{1,\ldots,6\}$ such that either $u_i$ cuts the equator in the first turn in the sector symmetric to $Q_i$, or $s_i$ cuts the equator in the first turn in the sector symmetric to $Q_{i-1}$.}

To simplify the exposition, we assume the directions of the vector field are those shown in Figure~\ref{fig:3}. We call the northern hemisphere to be the bounded one in the figure, and southern to the unbounded. 

We will divide the proof of this claim in several cases depending on the intersections of the stable and unstable varieties of the critical points at the equator. Recall that $u_1$ does not intersect $Q_1$ and if $u_1$ intersects $Q_2$, it will be its first intersection with $Q$. Moreover, if $u_1$ intersects $Q_4$ in the first turn, the claim follows, so we will not consider the possibility in which the first intersection of $u_1$ with $Q$ is at $Q_4$. Finally, also recall that because of the direction of the vector field, the first intersection of $u_1$ with $Q$ can not be at $Q_3$ or $Q_5$. Consequently, the following possibilities remain:
\begin{enumerate}
\item $u_1$ intersects $Q_2$, but then it does not intersect $Q_3$ in the first turn around the center.

\item $u_1$ intersects $Q_2$ and $Q_3$ in the first turn around the center.

\item $u_1$ does not intersect $Q_2$ in the first turn around the center.
\end{enumerate}

\medskip 

{\bf Case 1. $u_1$ intersects $Q_2$, but then it does not intersect $Q_3$ in the first turn around the center.}
The variety $u_1$ is in the southern hemisphere from $p_1$ to its first intersection with $Q_2$, and does not intersect $Q_3$ by hypothesis and $Q_4$ because of the direction of the vector field, remaining in the northern hemisphere. By Proposition~\ref{prop:sym}, $u_4$ is symmetric to $u_1$, so it remains in the northern hemisphere until it cuts $Q_5$, and then it can not intersect $Q_6$ or $Q_1$, staying in the southern hemisphere (see Figure~\ref{fig:4}). 

\begin{figure}[h]
\includegraphics[scale=.45]{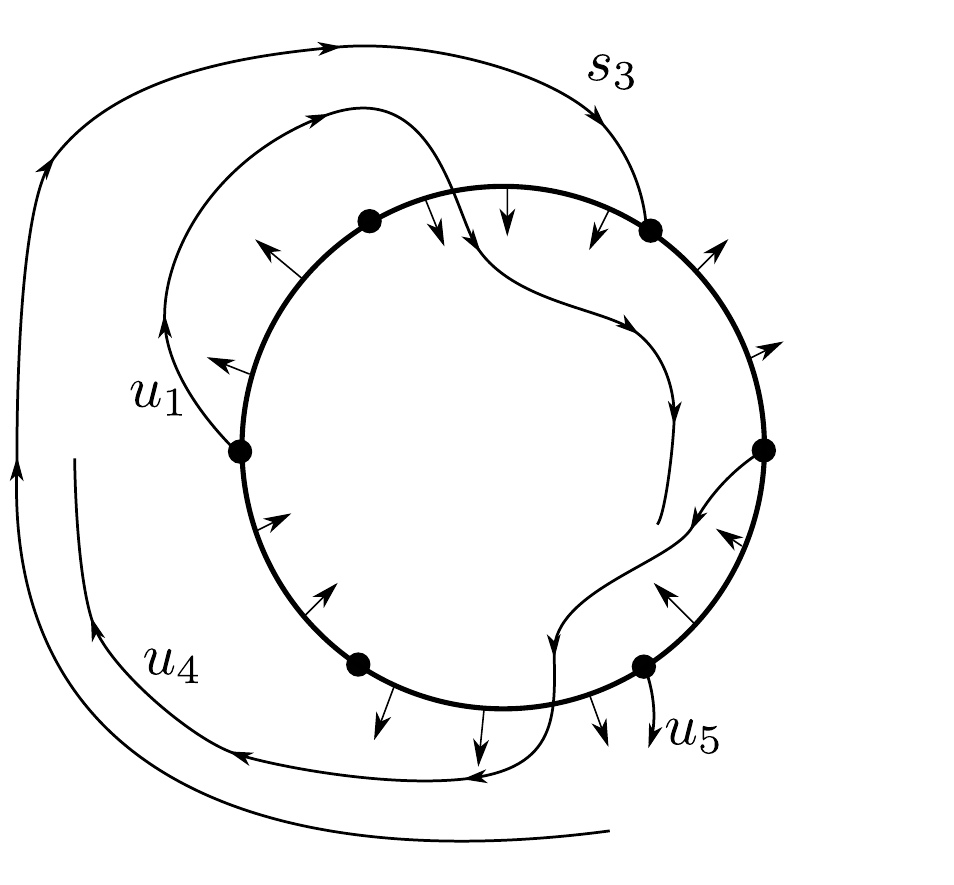}
\caption{Relevant cusp varieties in Case 1 in stereographic projection.}\label{fig:4}
\end{figure}

On the other hand, the variety $s_3$, in its first turn, can not cut the equator in the sectors $Q_2$, $Q_1$, $Q_6$. If it cuts the equator in $Q_5$ we prove the claim. If it does not cut $Q_5$, then $u_5$ is bounded by $s_3,$ $u_1$ and $u_4$, so it must cut $Q_2$, and the claim follows, see again Figure~\ref{fig:4}.

\medskip

{\bf Case 2. $u_1$ intersects $Q_2$ and then it intersects $Q_3$ in the first turn around the center.}
As in the previous case, the variety $u_1$ is in the southern hemisphere from $p_1$ to its first intersection with $Q_2$, and then comes back to the southern hemisphere when it intersects $Q_3$. In this situation, its symmetric variety $u_4$ remains in the northern hemisphere until it intersects $Q_5$ and then returns to that hemisphere at $Q_6$ (see figure~\ref{fig:2}.)

\begin{figure}[h]
\includegraphics[scale=.55]{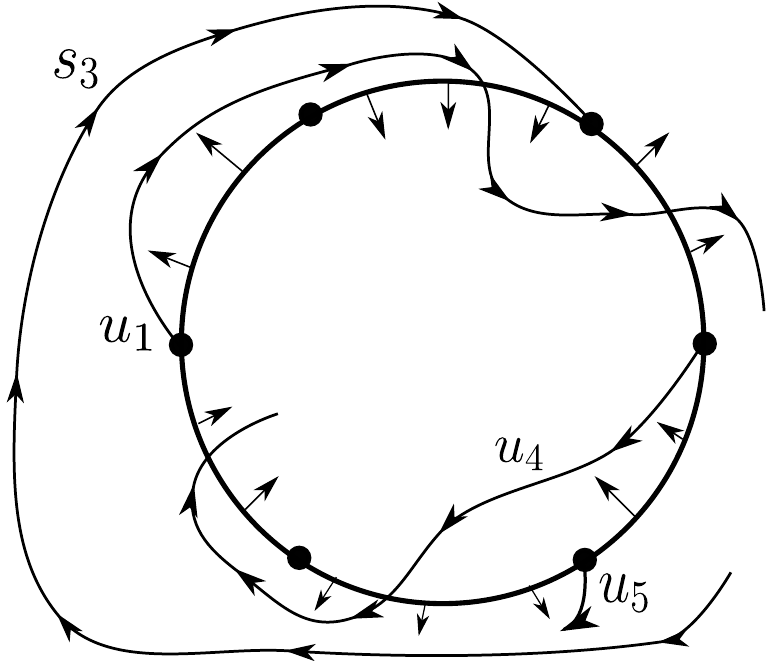}
\caption{Relevant cusp varieties in Case 2 in stereographic projection.}\label{fig:2}
\end{figure}

On the other hand, the variety $s_3$ starts in the southern hemisphere and can not intersect $Q_1$ because of the position of $u_1$, or $Q_2,Q_6$ because of the direction of the vector field. If it intersects $Q_5$, then the claim is proved, so we will assume it does not intersect $Q_5$. In this situation, we have three possibilities depending on the behaviour of $u_5$ (see Figure~\ref{fig:8}).
The variety $u_5$ starts in the southern hemisphere and can not intersect $Q_5$. There are the following possibilities:

\begin{enumerate}
    \item If $u_5$ does not intersect $Q_6$, then, because of the positions of $u_1$ and $s_3$ it must intersect $Q_2$, and the claim follows. 
    \item If $u_5$ intersects $Q_6$, and then it intersects $Q_1$, then because of the position of $u_1$ it must intersect $Q_2$, and the claim follows. 
    \item If $u_5$ intersects $Q_6$, and then it does not intersect $Q_1$, $u_5$ is in the already-proven Case 1 (relabeling it as $u_1$). 
\end{enumerate}

\begin{figure}[h]
\includegraphics[scale=.30]{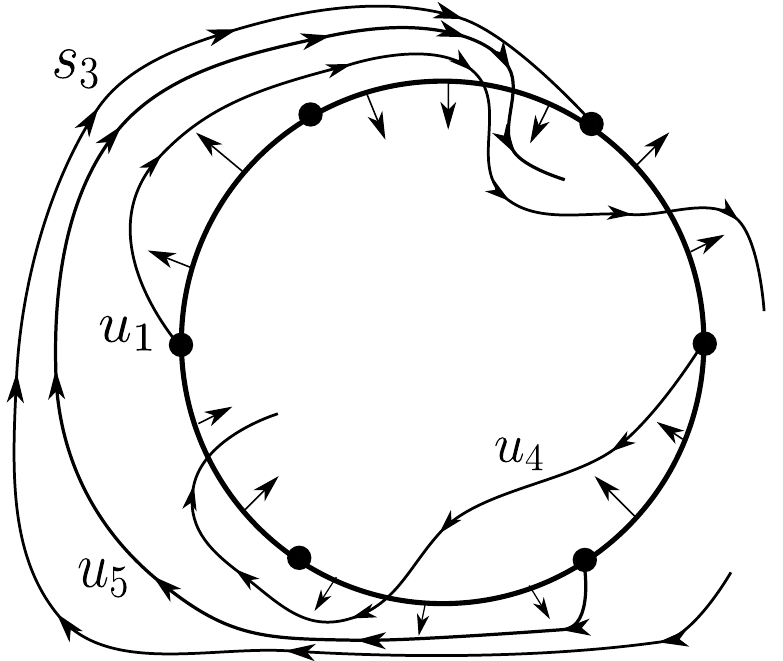}
\includegraphics[scale=.30]{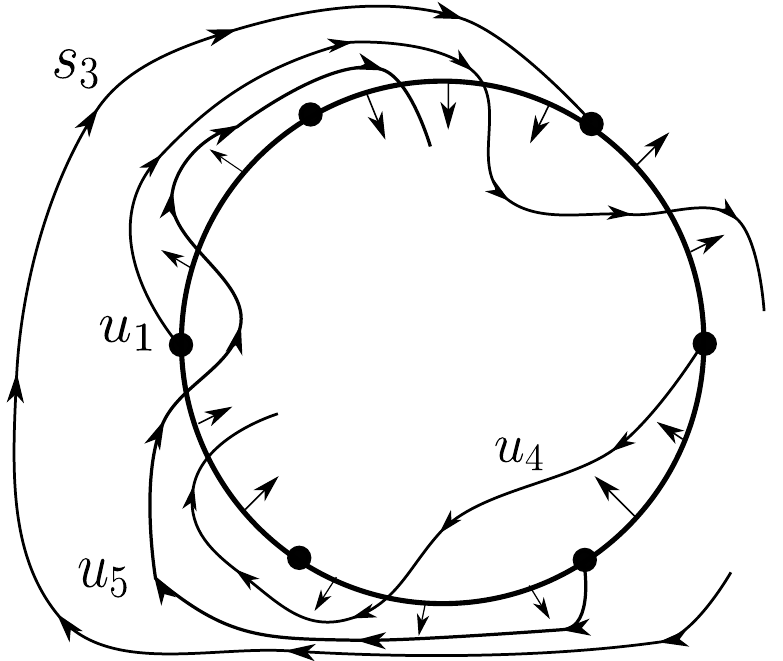}
\includegraphics[scale=.30]{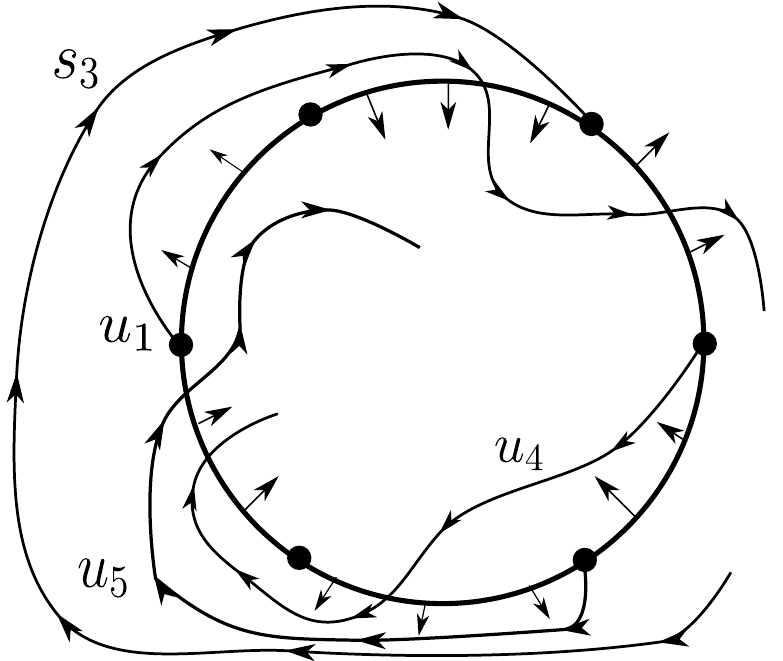}
\caption{Possibilities in Case 2 regarding the behaviour of $u_5$}\label{fig:8}
\end{figure}

\medskip 

{\bf Case 3. $u_1$ does not intersect $Q_2$ in the first turn around the center.}

In this case, the variety $u_1$ does not intersect $Q_2$ or $Q_4$ by hypothesis, nor $Q_1$, $Q_3$ or $Q_5$ because of the direction of the vector field.

But then $s_3$, that can not intersect $Q_2$ by the direction of the vector field, must intersect $Q_1$ because it is bounded by $u_1$. Consequently, reversing time, $s_3$ falls in Case 1 or Case 2, and the claim follows (see Figure~\ref{fig:9}). 

\begin{figure}[h]
\includegraphics[scale=.55]{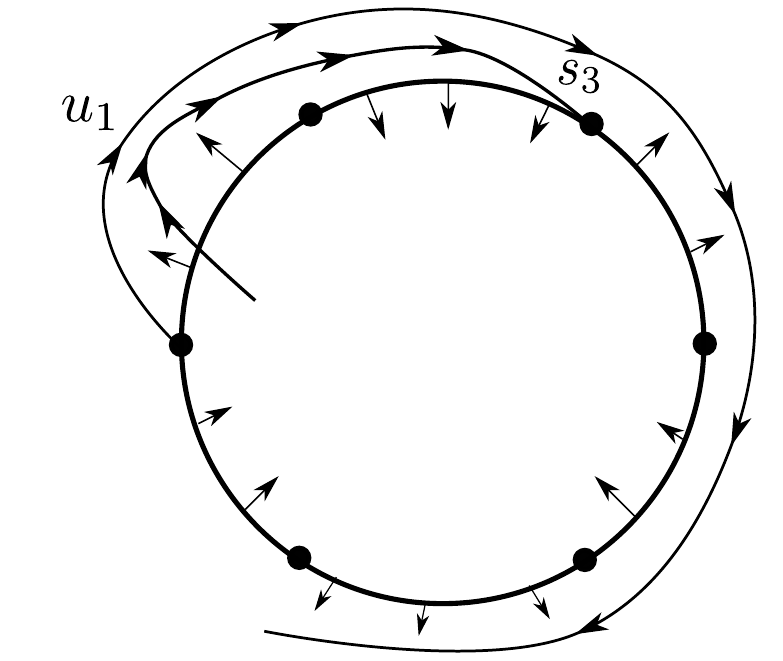}
\caption{Relevant cusp varieties in Case 3 in stereographic projection.}\label{fig:9}
\end{figure}

\medskip

We have proved Claim 1, and now we can proceed with the second part of the proof.

{\bf Claim 2. There exists a periodic solution crossing the equator.}

Without lack of generality, reversing time if necessary, we may assume that $u_1$ intersects $Q_4$ in the first turn. 

Assume first that $u_1$ does not intersect the equator before $Q_4$ and denote by $c_1$ the intersection point of $u_1$ and $Q_4$.

We have to distinguish two cases according to whether  $s_4$ intersects $Q_1$ or not. 

If $s_4$ intersects $Q_1$ in a point $c_4$, then any solution starting in the arc limited by $p_1,c_4$ must intersect the arc limited by $p_4,c_1$ (eventually passing through the northern hemisphere), except for the point corresponding to the intersection of $s_3$ with $Q_1$, but it can be extended in a continuous way to that point. See Figure~\ref{fig:5}.
\begin{figure}[h]
\includegraphics[scale=.30]{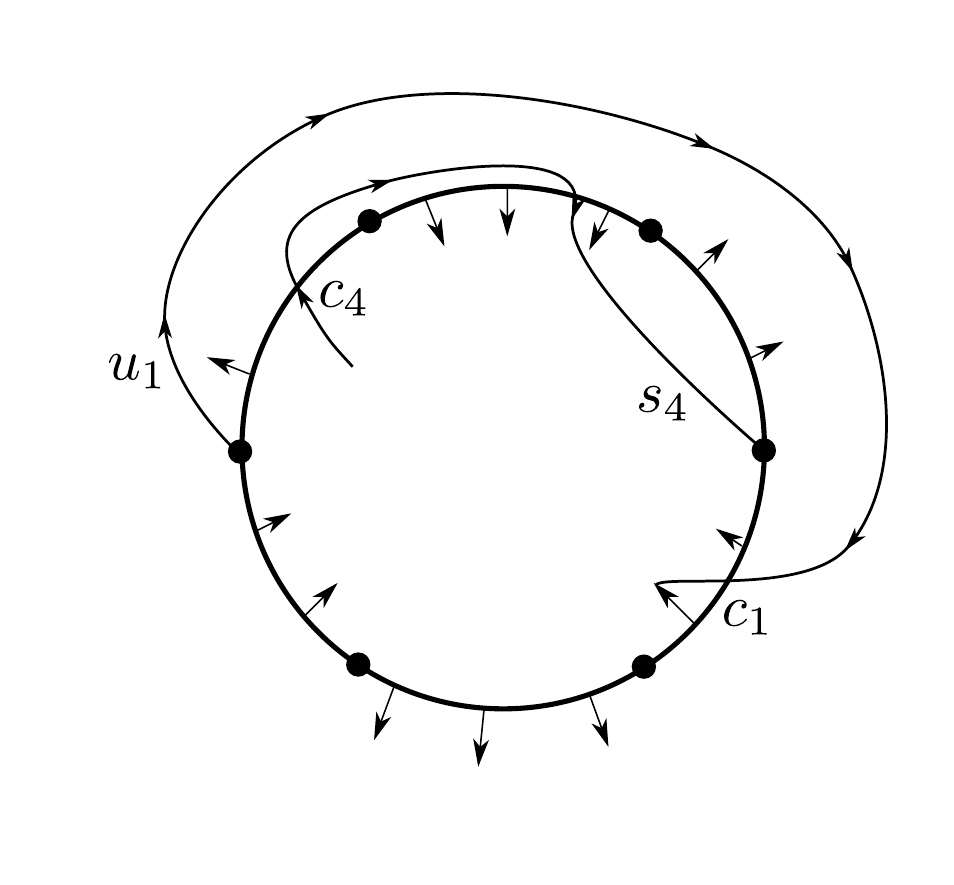}
\includegraphics[scale=.30]{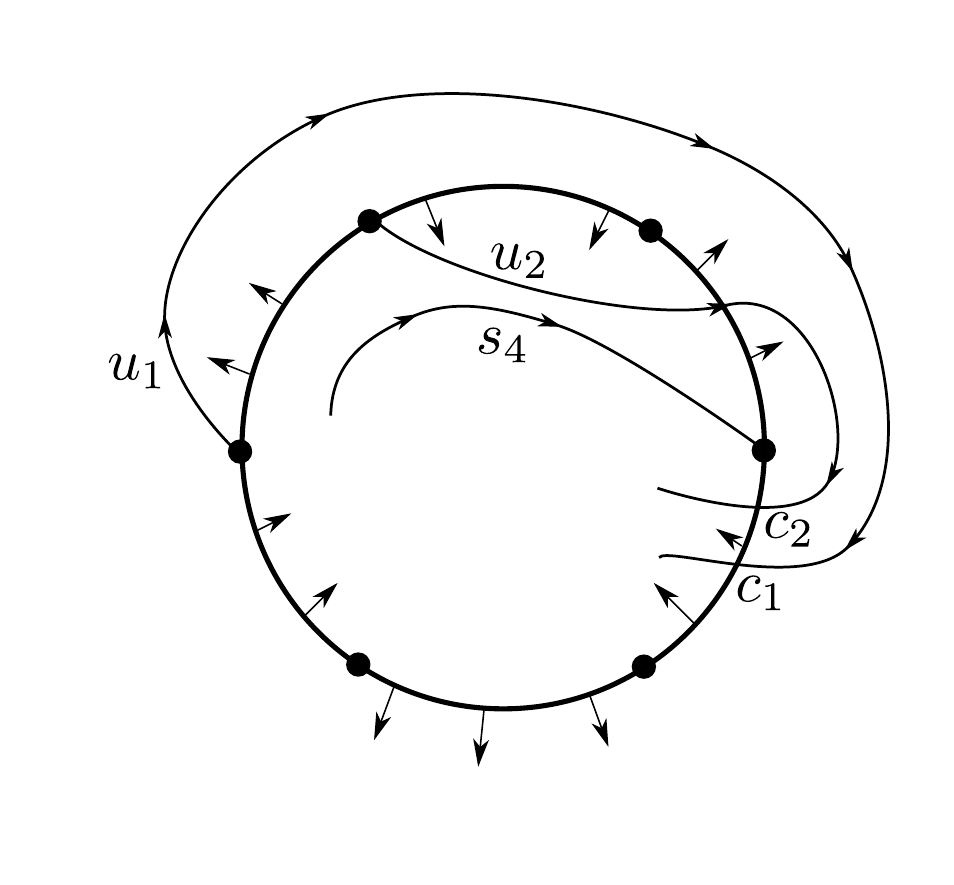}
\caption{Relative positions of $u_1$ and $s_3$.}\label{fig:5}
\end{figure}

Therefore, we may define a continuous map $\phi$ from the first arc into the second, such that 
\[
\lim_{p\to p_1}\phi(p)=c_1,\quad \lim_{p\to c_4} \phi(p) = p_4.
\]
As the angle of the sector limited by $p,\phi(p)$ is greater than $\pi$ for $p$ close to $c_1$ and lower than $\pi$ for $p$ close to $c_4$, there there exists a point $p$ such that the angle is $\pi$, and we conclude by Proposition~\ref{prop:sym}, except if $p$ correspond with a point in $s_3$. But in that case, again by Proposition~\ref{prop:sym}, there are two heteroclinic connections between $p_3$ and $p_6$.

Now, assume that $s_4$ does not intersect $Q_1$ in a point. Then $u_2$ must intersect $Q_4$ in a point $c_2$. Therefore, a solution starting at a point of $Q_1$ intersects $Q_4$ (in the arc defined by $c_1,c_2$). Define a map $\phi\colon Q_1\to Q_4$ as $\phi(p)$ the first intersection of the solution starting at $p$ with $Q_4$. Composing this map with the symmetry with respect to the origin and applying Brower's fixed point theorem, we obtain $p\in Q_1$ such that $\phi(p)=-p$. By Proposition~\ref{prop:sym}, we conclude. 

Finally, if $u_1$ intersects $Q$ at one point previous to $Q_4$, it must be at $Q_2$, but in order to intersect $Q_4$ afterwards, it must intersect $Q_3$ as well. Note that in this situation, there is no obstruction to repeat the previous argument in exactly the same way, and complete the proof. 
\end{proof}

\begin{coro}\label{corolario_cl} Assuming Hypothesis~\eqref{hypo},
there are systems of the form \eqref{eq:campoesfera} with at least three limit cycles in the sphere.
\end{coro}
\begin{proof}
    If we compactify system \eqref{eq:cl} through the Poincar\'e compactification, we get that there is at least one limit cycle in each one of both hemispheres. Now we will prove that for this system $D=a_2^2-4a_1a_3<0$ and, applying Theorem \ref{teo-po-esfera} we will conclude that another periodic orbit exists. Using Hypothesis \ref{hypo}, this periodic orbit will be a limit cycle and hence the system will have, at least, three limit cycles.

    If we compute $D$ for system \eqref{eq:cl} we get
    \[
    D=a_2^2-\frac{(-3+10a_2\pi+98\varepsilon)(1+120a_2\pi-82\varepsilon)}{74^2\pi^2}.
    \]
    We have to prove that, for $\varepsilon<0$ small enough, there exist values of $a_2$ such that $D<0.$ To do that, we will compute the discriminant of $D$ and show that, for $\varepsilon<0$ small enough, the discriminant is positive. The discriminant of $D$ is
    \[\Delta=\frac{8}{74^2\pi^2}(11+10432\varepsilon^2-678\varepsilon).
    \]
    Hence, for $\varepsilon<0$ the previous  discriminant is positive, and therefore, there  exist values of $a_2$ such that $D<0$ and we can apply Theorem \ref{teo-po-esfera}.
\end{proof}

\subsection{Invariant lines}

Invariant lines of the planar vector field \eqref{eq:main} correspond in the compactified vector field~\eqref{eq:campoesfera} to either heteroclinic connections, if they intersect the equator in critical points, or to maximal circles, if not. 
In this subsection we will prove that the last case can not happen, and therefore there are systems inside family \eqref{eq:campoesfera} having heteroclinic connections. 

If $q(x,y)=c_1x+c_2y+c_3$ is an invariant line, then $c_3\neq0$, since it can not pass through the critical point, and it must satisfy the relationship
\[q_x(x,y)P(x,y)+q_y(x,y)Q(x,y)=q(x,y)K(x,y),\]
where $x'=P(x,y), y'=Q(x,y)$ and for some cofactor $K(x,y).$ 

Therefore, a simple way to obtain the conditions to have an invariant line is by computing the remainder of  $q_xP+q_yQ$ divided by $q$, considered both as polynomials in $x$. Now, equaling the remainder to zero, we get that $c_1 = -b_2 c_3$ (induced by the assumption $a_4=0$), and
\[
a_1 = \frac{b_2^2 (c_2 -b_1 c_3 )}{c_3}, \quad
a_2 = -\frac{2 c_2  b_2 (c_2 -b_1 c_3 )}{c_3^2}, \quad
a_3= \frac{c_2^2 (c_2 -b_1 c_3 )}{c_3^3}.
\] 
Now, eliminating $c_2,c_3$, we obtain the conditions for \eqref{eq:main} to have an invariant line
\begin{equation}\label{condinv}
a_2 b_2^3 + 2a_1\left(a_1+ b_1
   b_2^2\right)=0,\quad 
a_3 b_2^6- a_1 \left(a_1+b_1
   b_2^2\right)^2=0.
\end{equation}
Moreover, there is exactly one invariant line when the previous equations hold and its expression is given by
\begin{equation}\label{recinv}
-b_2^3 x + (a_1 + b_1 b_2^2) y + b_2^2 =0.
\end{equation}

We remind that in the chart $\mathcal{U}_2$, the infinite critical points are $(\hat{u},0)$, with $\hat{u}$ being a root of the cubic $g(u)=-u(a_1 u^2+a_2 u+a_3)$. 

It is immediate to check that the invariant line \eqref{recinv} in coordinates $(u,v)$ always crosses some of the infinite critical points when the conditions \eqref{condinv} are satisfied. Furthermore, under \eqref{condinv} we have $D=a_2^2-4a_1a_3=0$, so in this case the infinite critical points through which the invariant line crosses the equator are degenerate. 

As a conclusion, the invariant lines of the planar vector field correspond always to heteroclinic connections. 

 If this was not the case, {\it i.e.} if in some cases the invariant straight lines did not pass through a critical point, then they would constitute a periodic orbit. Hence, the results in \cite{Grau} could be applied in order to know whether the existing algebraic periodic orbit (proved in Theorems \ref{teo-po-esfera} and \ref{teo-po-esfera2}) is a limit cycle. But, as proved in the current section, this is never the case.

\section{Conclusions and open questions}\label{Sec4}

In this paper we have studied the rigid quartic family of planar vector fields \eqref{eq:rigido} with $F(x,y)$ denoting the function defined in expression \eqref{funcionF}. Within this family, we have stated and proved the conditions that determine the existence of a center, and additionally, we have identified a subfamily characterized by the absence of limit cycles in the plane.

Under our point of view, the main contribution of the article lies in the study of the rigid system by means of its associated system \eqref{eq:campoesfera}, which is defined in the Poincar\'e sphere. By using local charts, we have characterized the infinite critical points of the  rigid system, classifying them as either cusps or as the  union of two hyperbolic and two parabolic sectors.

Despite all the previous proven results, there are still some open problems that should be solved in order to have a complete understanding of the studied family.

The main question that remains open deals with the number of limit cycles of system \eqref{eq:campoesfera}. 

On the one hand, we would need to determine the number of limit cycles that do not cross the equator. Those limit cycles would then be confined into one of the hemispheres and should also be  limit cycles of the finite system \eqref{eq:main}.
In this paper we have found examples inside the family \eqref{eq:main} with at least one limit cycle. Furthermore, we have conducted several numerical experiments indicating that the maximum number of limit cycles in the finite family \eqref{eq:main} is one. Thus, our first open question would be the following one:

\begin{conj} Is one the maximum number of limit cycles of system \eqref{eq:main}?
\end{conj}

On the other hand, we have to deal with the limit cycles that intersect the equator. In Theorems \ref{teo-po-esfera} and \ref{teo-po-esfera2} we have proved that when all the critical points in the equator are cusps, system \eqref{eq:campoesfera} always has either a periodic orbit intersecting the equator or a heteroclinic connection. Some questions arise from this result. 

The first one deals with Hypothesis \ref{hypo}. As demonstrated in Theorem \ref{theo:centersphere}, we have established that all finite centers (around the origin), when they exist, are global centers. However, for system \eqref{eq:campoesfera}, the reciprocity question remains unsolved. Specifically, we have adopted it as Hypothesis \ref{hypo}, constituting the next open question we pose:
\begin{conj} If system \eqref{eq:campoesfera} has an annulus or periodic orbits, is it a global center?
\end{conj}
Note that if the answer to this question is affirmative, then this periodic orbit that intersects the equator, when it exists, will be a limit cycle of \eqref{eq:campoesfera}, except for the cases of global centers stated in Theorem \ref{theo:center}. 

Also related to the limit cycles in the sphere, in Corollary \ref{corolario_cl} we have proved that, if the previous open question has an affirmative answer, there are systems inside family \eqref{eq:campoesfera} having at least three limit cycles. Hence, we have a lower bound for the number of limit cycles for this family. The next question, much more difficult, has to deal with the upper bound:
\begin{conj}
    Which is the maximum number of limit cycles that system \eqref{eq:campoesfera} can have?
\end{conj}

Another natural question would be to check if the same result proved in Theorems \ref{teo-po-esfera} and \ref{teo-po-esfera2} is also true even when the infinite critical points are not all cusps. This turns out to be our next open question:
\begin{conj}
    Does system \eqref{eq:campoesfera} always have  either a periodic orbit that intersects the equator or a heteroclinic connection?
\end{conj}

Recall that we have proved that there exist invariant straight lines for system \eqref{eq:campoesfera} and that, in all the cases they exist, they form heteroclinic connections between the infinite critical points. The fact that in these cases the heteroclinic connections are algebraic suggests to us that the existence of this kind of connections is very unlikely. This leads us to the following open question:
\begin{conj}
    Do the systems inside family \eqref{eq:campoesfera} having heteroclinic connections constitute a zero measure set in the set of parameters?
\end{conj}

In summary, if open questions 2, 4 and 5 have an affirmative answer, then we would have proved that system \eqref{eq:campoesfera} always has at least one limit cycle that crosses the equator, except for a zero measure set in the set of parameters. 

Finally, we can present a last open question related to the nature of the periodic orbit crossing the equator. 
The invariant straight lines of system \eqref{eq:campoesfera} studied in Section 3.5 are always heteroclinic connections, never periodic orbits. This fact leads us to ask the following question. 

\begin{conj}
Can the periodic orbit that crosses the equator be algebraic?
\end{conj}

\section*{Acknowledgments}

The authors are partially supported by grant number PID2020-118726GB-I00 funded by MCIN/AEI/10.13039/501100011033.


\begin{thebibliography}{99}

\bibitem{ARJMAA03}   A. Algaba, M. Reyes, {\em Computing center conditions for vector fields with constant angular speed.}
Journal of Computational and Applied Mathematics, {\bf 154}, (2003), 143--159.

\bibitem{Andreev} A.F. Andreev. {\em Investigation of the behaviour of the integral curves of a system
of two differential equations in the neighbourhood of a singular point.} Trans.
Amer. Math. Soc., {\bf 8}, (1958), 183-–207, .


\bibitem{BravoFerTer} J.L. Bravo, M. Fern\'andez, A. Teruel, {\em  Poincaré compactification for non-polynomial vector fields}, Qualitative Theory of Dynamical Systems, {\bf 19}, (2020), 50.

\bibitem{BT} J.L. Bravo, J. Torregrosa, {\em Abel-like equations with no periodic solutions}, J. Math. Anal. Appl., {\bf 342}, (2008), 931--942.

\bibitem{Conti} R. Conti, {\em Uniformly isochronous centers of polynomial systems in $\mathbb{R}^2$}, Differential Equations, (1994), Taylor \& Francis.

\bibitem{Dumo} F. Dumortier, J. Llibre, J. C. Artés, {\em Qualitative theory of planar differential systems.}  (2006), Universitext, Springer-Verlag.


\bibitem{Gasull} A. Gasull,  {\em Some open problems in low dimensional dynamical systems}, SeMA Journal, {\bf 78}, (2021), 233-269.

\bibitem{GasProTor} A. Gasull, R. Prohens, J. Torregrosa,  {\em Limit cycles for rigid cubic systems }, J. Math. Anal. Appl., {\bf 303}, (2005), 391-404.

\bibitem{GasTor} A. Gasull, J.Torregrosa. {\em Some results on rigid systems.} In International Conference on Differential Equations (Equadiff-2003), 340-345. World Sci. Publ., Hackensack, NJ, 2005.

\bibitem{Grau} H. Giacomini, M. Grau, {\em On the stability of limit cycles for planar differential
systems}, J. Diff. Eq., {\bf 213}, (2005), 368-388.

\bibitem{Itikawa_tesis} J. Itikawa. {\em Uniform isochronous centers of degrees 3 and 4 and their perturbations.} Universitat Autònoma de Barcelona (2015).

\bibitem{Itikawa} J. Itikawa, J. Llibre, {\em Phase portraits of uniform isochronous quartic centers}, Journal of Computational and Applied Mathematics, {\bf 285} (2015), 98--114.

\bibitem{LlibRamRamSad} J. Llibre, R. Ram\'irez, V. Ram\'irez, N. Sadovskaia, {\em Centers and uniform isochronous centers of planar polyonmial differential systems}, J. Dyn. Diff. Equat., {\bf 30} (2018), 1295-1310.

 \bibitem{Rudenok} A.E. Rudenok, {\em Strong isochronism of a center. Periods of limit cycles of a Lienard’s system}, Differential Equations, {\bf 11}, (1975), 610–-617.


\end{thebibliography}
\end{document}